\documentclass{article}
\usepackage{amsmath,amsthm,amssymb,mathrsfs,amstext, titlesec,enumitem}
\makeatletter

\titleformat{\section}
{\normalfont\Large\bfseries}{\thesection}{1em}{}
\titleformat*{\subsection}{\Large\bfseries}
\titleformat*{\subsubsection}{\large\bfseries}
\titleformat*{\paragraph}{\large\bfseries}
\titleformat*{\subparagraph}{\large\bfseries}

\renewcommand{\@seccntformat}[1]{\csname the#1\endcsname. }

\renewenvironment{abstract}{%
    \if@twocolumn
      \section*{\abstractname}%
    \else 
      \begin{center}%
        {\bfseries \Large\abstractname\vspace{\z@}}
      \end{center}%
      \quotation
    \fi}
    {\if@twocolumn\else\endquotation\fi}
    
\makeatother
\theoremstyle{plain}
\newtheorem{thm}{Theorem}
\newtheorem{lem}[thm]{Lemma}
\theoremstyle{definition}

\newtheorem{Conj}[thm]{Conjecture}
\providecommand{\keywords}[1]{{\bf{Keywords:}} #1}
\providecommand{\subjectclass}[2]{\textbf{Mathematics subject classification 2020:} #1}

\title{Additive and multiplicative Gower's Ramsey theorem}

\author{Sayan Goswami \\ \textit{sayangoswami@imsc.res.in} \\ The Institute of Mathematical Sciences\\ A CI of Homi Bhabha National Institute\\ CIT Campus, Taramani, Chennai 600113, India.}

\date{\vspace{-5ex}}

\begin{document}
\maketitle
\begin{abstract}
In \cite[Theorem 1]{key-3}, W. T. Gower generalized Hindman's Finite sum theorem over
$X_{k}=\left\{ \left(n_{1},n_{2},\ldots,n_{k}\right):n_{1}\neq0\right\} $
by showing that for any finite coloring of $X_{k}$ there exists a
sequence such that the Gower subspace generated by that sequence is
monochromatic. For $k=1,$ this immediately gives the finite sum theorem.
In this article we will show that for any finite coloring of $X_{k}$
there exist two sequences $\left\{ \mathbf{n_{i}}:i\in I\right\} $
and $\left\{ \mathbf{m_{i}}:i\in I\right\} $ such that the Gower
subspace generated by $\left\{ \mathbf{n_{i}}:i\in I\right\} $ and
set of all finite products of $\left\{ \mathbf{m_{i}}:i\in I\right\} $
are in a single color. This immediately generalize a result of V.
Bergelson and N. Hindman \cite[Theorem 2.4]{key-1}, which says that for any finite coloring
of $\mathbb{N}$, there exist two sequences $\left(x_{n}\right)_{n}$
and $\left(y_{n}\right)_{n}$ such that the finite sum and product
generated by $\left(x_{n}\right)_{n}$ and $\left(y_{n}\right)_{n}$
are in a same color.
\end{abstract}

\subjectclass{05D10}\\

\keywords Ramsey theory, Gower's theorem \\

One of the major type of study in Ramsey theory deals with monochromatic
structure in finitely colored spaces. Hindman's Finite sum theorem \cite{key-4}
is one of the cornerstone result in Ramsey theory, which says that
for any finite partition of $\mathbb{N}$, there exists a sequence
$\left(x_{n}\right)_{n}$ such that the set of all finite sums of
$\left(x_{n}\right)_{n}$, say $FS\left(\left(x_{n}\right)_{n}\right)=\left\{ \sum_{n\in H}x_{n}:H\in\mathcal{P}_{f}\left(\mathbb{N}\right)\right\} $
is monochromatic. Similar result is true if we consider the product
operation instead of summation. It can be shown that for any finite
coloring of $\mathbb{N}$, we can't get a sequence $\left(x_{n}\right)_{n}$
such that the $FS\left(\left(x_{n}\right)_{n}\right)$ and $FP\left(\left(x_{n}\right)_{n}\right)$
are both in a single color. In \cite[Theorem 2.4]{key-1}, V. Bergelson and N. Hindman
proved that for any finite coloring of $\mathbb{N}$, there exist
two sequences $\left(x_{n}\right)_{n}$ and $\left(y_{n}\right)_{n}$
such that the $FS\left(\left(x_{n}\right)_{n}\right)$ and $FP\left(\left(y_{n}\right)_{n}\right)$
are both in a single color. In \cite[Theorem 1]{key-3}, W. T. Gower provided a generalization
of the Hindman theorem using methods of ultrafilters. To state his result,
we need some basic preliminaries. Let $\mathbb{N}_{0}=\mathbb{N}\cup\left\{ 0\right\} $
and for any $k\in\mathbb{N}$, define the shift map $T:\mathbb{N}_{0}^{k}\rightarrow\mathbb{N}_{0}^{k}$
by $T\left(n_{1},n_{2},\ldots,n_{k}\right)=\left(0,n_{1},n_{2},\ldots,n_{k-1}\right)$.
Let $X_{k}=\mathbb{N}_{0}^{k}\setminus T\mathbb{N}_{0}^{k}=\left\{ \left(n_{1},n_{2},\ldots,n_{k}\right):n_{1}\neq0\right\} $.
Given any subset $A=\left\{ \mathbf{n_{i}}:i\in I\right\} \subset X_{k}$,
the Gower sum subspace generated by $A$ is the set of elements of $\mathbb{N}_{0}^{k}$
of the form $\sum_{j=1}^{k}\sum_{i\in B_{j}}T^{j-1}\mathbf{n_{i}}$,
where $B_{1},B_{2},\ldots,B_{k}$ are disjoint subsets of $I$ and
$B_{1}$is nonempty.
\begin{thm}
If $k,r\in\mathbb{N}$ and $X_{k}=\bigcup_{i=1}^{r}C_{i}$ be a $r$-partition
of $X_{k}$, then there exist a sequence $A=\left\{ \mathbf{n_{i}}:i\in I\right\} \subset X_{k}$
such that the Gower sum subspace generated by $A$ is monochromatic.
\end{thm}

Let us now describe some ultrafilter preliminaries. For any set $X$,
let $U\left(X\right)$ be the set of all ultrafilters on $X$. If
$\alpha\in U\left(X\right)$ , the symbol $\Lambda_{\alpha}$ is defined
as follows. If $P\left(x\right)$ is any proposition defined over
the set $X$, then by $\left(\Lambda_{\alpha}x\right)P\left(x\right)$
we mean $\left\{ x\in X:P\left(x\right)\text{ is true}\right\} \in\alpha.$
If $\left(X,\cdot\right)$ is a semigroup, then $\left(U\left(X\right),\cdot\right)$
is a compact semigroup where for $\alpha,\beta\in U\left(X\right)$,
$\alpha\cdot\beta=\left\{ A\subseteq X:\left(\Lambda_{\alpha}x\right)\left(\Lambda_{\beta}y\right)\left(x\cdot y\right)\in A\right\} .$
Note that for any semigroup $S$ and its subsemigroup $T,$ the embedding
map $i:T\rightarrow S$ can be extend continuously as $i:U\left(T\right)\rightarrow U\left(S\right)$.

Given $k\geq2,$ the induced shift map $T^{*}:U\left(X_{k}\right)\rightarrow U\left(X_{k-1}\right)$
is defined as follows. For any $\alpha\in U\left(X_{k}\right),$ $T^{*}\left(\alpha\right)=\left\{ A\subseteq X_{k-1}:\left(\Lambda_{\alpha}x\right)Tx\in A\right\} .$
This map $T^{*}$ is continuous. Given any $j<k\in\mathbb{N},$ define
$+:\left(X_{j},X_{k}\right)\rightarrow X_{k}$ by 
\[
\left(m_{1},\ldots,m_{j}\right)+\left(n_{1},\ldots,n_{k}\right)=\left(n_{1},\ldots,n_{k-j},m_{1}+n_{k-j+1},\ldots,m_{j}+n_{k}\right).
\]
 Similarly one can define $+:\left(X_{k},X_{j}\right)\rightarrow X_{k}$

This map induces the map $+:\left(U\left(X_{k}\right),U\left(X_{j}\right)\right)\rightarrow U\left(X_{k}\right)$
defined by 
\[
\alpha+\beta=\left\{ A\subseteq X_{k}:\left(\Lambda_{\alpha}x\in X_{k}\right)\left(\Lambda_{\beta}y\in X_{j}\right)\left(x+y\right)\in A\right\} .
\]
 Similarly define the shift map $S:\mathbb{N}^{k}\rightarrow\mathbb{N}^{k}$
by $S\left(n_{1},n_{2},\ldots,n_{k}\right)=\left(1,n_{1},n_{2},\ldots,n_{k-1}\right)$.
Let $Y_{k}=\mathbb{N}^{k}\setminus S\mathbb{N}^{k}=\left\{ \left(n_{1},n_{2},\ldots,n_{k}\right):n_{1}\neq1\right\} $.
Given any subset $B=\left\{ \mathbf{m_{i}}:i\in I\right\} \subset Y_{k}$,
the Gower product subspace generated by $B$ is the set of elements of $\mathbb{N}^{k}$
of the form $\prod_{j=1}^{k}\prod_{i\in B_{j}}S^{j-1}\mathbf{m_{i}}$,
where $B_{1},B_{2},\ldots,B_{k}$ are disjoint subsets of $I$ and
$B_{1}$is nonempty. Similarly we can define $S^{*}:U\left(Y_{k}\right)\rightarrow U\left(Y_{k-1}\right)$,
where the operation will be multiplicative. Both the maps are right
continuous. For details on ultrafilters the readers can see the nice textbook \cite{key-5}.

The main objective to prove Gower's theorem was to find out an idempotent
ultrafilter $\alpha$ in $U\left(X_{k}\right)$ such that $\left(T^{*}\right)^{j}\alpha+\alpha=\alpha+\left(T^{*}\right)^{j}\alpha=\alpha$
for each $0\leq j\leq k-1$.

The following lemma is refined version of \cite[Lemma 3]{key-3}.
\begin{lem}
For every $k\in\mathbb{N},$ there exists an ultrafilter $\alpha\in U\left(X_{k}\right)$
such that $\left(T^{*}\right)^{j}\alpha+\alpha=\alpha+\left(T^{*}\right)^{j}\alpha=\alpha$
for each $0\leq j\leq k-1$ and $\mathbb{N}^{k}\in\alpha$.
\end{lem}

\begin{proof}
The proof is similar to the proof of \cite[Lemma 3]{key-3}, except that we have to assume that
extra assumption in the induction hypothesis, so we leave the proof
to the reader.
\end{proof}
\begin{lem}
For every $k\in\mathbb{N},$ the set 
\[
\overline{\left\{ \alpha:\left(T^{*}\right)^{j}\alpha+\alpha=\alpha+\left(T^{*}\right)^{j}\alpha=\alpha \, \text{for each}\,0\leq j\leq k-1 \right\} }
\]
is a left ideal of $\left(\beta\mathbb{N}^{k},\cdot\right).$
\end{lem}

\begin{proof}
Let $\varGamma=\overline{\left\{ \alpha:\left(T^{*}\right)^{j}\alpha+\alpha=\alpha+\left(T^{*}\right)^{j}\alpha=\alpha\right\} }$.
If we can show that $\mathbb{N}^{k}\cdot\Gamma\subseteq\Gamma,$ then
we have the required result. Let $\gamma\in\Gamma$. Now $A\in\mathbf{\bar{n}}\gamma$
if and only if $\mathbf{\bar{n}}^{-1}A\in\gamma.$ Pick $\alpha\in U\left(X_{k}\right)$
such that $\left(T^{*}\right)^{j}\alpha+\alpha=\alpha+\left(T^{*}\right)^{j}\alpha=\alpha$
and $\mathbf{\bar{n}}^{-1}A\in\alpha.$ Now 
\[
C\in\bar{\mathbf{n}}\alpha+\left(T^{*}\right)^{j}\bar{\mathbf{n}}\alpha
\]
iff 
\[
\left(\Lambda_{\bar{\mathbf{n}}\alpha}x\right)\left(\Lambda_{\left(T^{*}\right)^{j}\left(\bar{\mathbf{n}}\alpha\right)}y\right)\left(x+y\right)\in C
\]
iff 
\[
\left(\Lambda_{\bar{\mathbf{n}}\alpha}x\right)A\left(\Lambda_{\bar{\mathbf{n}}\alpha}z\right)\left(x+\left(T^{*}\right)^{j}z\right)\in C
\]
iff
\[
\left\{ x:\left\{ z:\left(x+\left(T^{*}\right)^{j}z\right)\in C\right\} \in\bar{\mathbf{n}}\alpha\right\} \in\bar{\mathbf{n}}\alpha
\]
iff
\[
\bar{\mathbf{n}}^{-1}\left\{ x:\bar{\mathbf{n}}^{-1}\left\{ z:\left(x+\left(T^{*}\right)^{j}z\right)\in C\right\} \in\alpha\right\} \in\alpha
\]
iff 
\[
\left\{ x':\left\{ z':\bar{\mathbf{n}}x'+\left(T^{*}\right)^{j}\bar{\mathbf{n}}z'\in C\right\} \in\alpha\right\} \in\alpha
\]
iff
\[
\left(\Lambda_{\alpha}x'\right)\left(\Lambda_{\alpha}z'\right)\left(\bar{\mathbf{n}}x'+\left(T^{*}\right)^{jA\in\delta}\bar{\mathbf{n}}z'\right)\in C
\]
iff 
\[
\left(\Lambda_{\alpha}x'\right)\left(\Lambda_{\alpha}z'\right)\left(x'+\left(T^{*}\right)^{j}z'\right)\in\bar{\mathbf{n}}^{-1}C
\]
iff
\[
\bar{\mathbf{n}}^{-1}C\in\alpha+\left(T^{*}\right)^{j}\alpha A\in\delta=\alpha
\]
iff 
\[
C\in\bar{\mathbf{n}}\alpha.
\]

Hence $\mathbf{\bar{n}}^{-1}A\in\alpha$ implies $A\in\mathbf{\bar{n}}\alpha$
implies $A\in\bar{\mathbf{n}}\alpha+\left(T^{*}\right)^{j}\bar{\mathbf{n}}\alpha$,
where $\left(T^{*}\right)^{j}\bar{\mathbf{n}}\alpha+\bar{\mathbf{n}}\alpha=\bar{\mathbf{n}}\alpha+\left(T^{*}\right)^{j}\bar{\mathbf{n}}\alpha=\bar{\mathbf{n}}\alpha.$
So, $\mathbf{\bar{n}}\gamma\in\Gamma.$
\end{proof}

\begin{thm}
Let $k,r\in\mathbb{N}$ and $X_{k}=\bigcup_{i=1}^{r}C_{i}$ be a $r$-partition
of $X_{k}$. Thus there exist two sequences $B=\left\{ \mathbf{m_{i}}:i\in I\right\} \subset X_{k}$
and $\langle y_{n}\rangle_{n\in\mathbb{N}}\subset\mathbb{N}^{k}$
such that the Gower sum subspace generated by $B$ and $FP\left(\langle y_{n}\rangle_{n\in\mathbb{N}}\right)$ are both in the same partition. 
\end{thm}

\begin{proof}
As $\varGamma$ is a left ideal of $\left(\beta\mathbb{N}^{k},\cdot\right),$
there will be an idempotent $\delta\in\varGamma$ (as left ideals are closed, this follows from \cite{key-2}). As $\mathbb{N}^{k}\in\delta,$
we have $C_i\cap\mathbb{N}^{k}\in\delta$ for some $i\in \{1,2,\ldots,r\}$. As $C_i\in\delta$, we have
an infinite subset $B=\left\{ \mathbf{m_{i}}:\mathbf{i}\in I\right\} $
which generate a subspace contained in $C_i$. Now $\delta$ is an multiplicative idempotent,
so $C_i\cap\mathbb{N}^{k}$ will contain configuration of the form $FP\left(\langle y_{n}\rangle_{n\in\mathbb{N}}\right)$
for some sequence $\langle y_{n}\rangle_{n\in\mathbb{N}}$ in $\mathbb{N}^{k}.$
This completes the proof.
\end{proof}
We conjecture that both of the Gower sum and product subspace will be monochromatic.

\begin{Conj}
Let $k,r\in\mathbb{N}$ and $X_{k}=\bigcup_{i=1}^{r}C_{i}$ be a $r$-partition
of $X_{k}$. Thus there exist two subsets $A=\left\{ \mathbf{n_{i}}:i\in I\right\} \subset X_{k}$
and $B=\left\{ \mathbf{m_{i}}:i\in I\right\} \subset X_{k}$ such
that the Gower sum subspace generated by $A$ and the product subspace generated
by $B$ are both in the same color.
\end{Conj}

\end{document}